\documentclass[11pt]{amsart}
\usepackage[sorted-cites]{amsrefs}
\usepackage{verbatim}
\usepackage{color}
\usepackage{amsmath,amsthm,amscd}
\usepackage{calrsfs}
\DeclareMathAlphabet{\pazocal}{OMS}{zplm}{m}{n}
\usepackage{lipsum}

\newtheorem{theorem}{Theorem}[section]

\newtheorem{lemma}{Lemma}[section]
\newtheorem{cor}{Corollary}[section]
\newtheorem{prop}{Proposition}[section]

\newtheorem{definition}{Definition}[section]

\newtheorem{example}{Example}[section]
\theoremstyle{remark}
\newtheorem{remark}{Remark}[section]
\numberwithin{equation}{section}

\renewcommand{\bar}{\overline}

\begin{document}

\title[Spectral properties and rigidity for self-expanding solutions]{Spectral properties and rigidity for self-expanding solutions  of  the mean curvature flows}



\thanks{The authors are partially supported by CNPq and Faperj of Brazil.}

\address{Instituto de Matem\'atica e Estat\'\i stica, Universidade Federal Fluminense,
Niter\'oi, RJ 24020, Brazil}

\author[Xu Cheng]{Xu Cheng}

\author[Detang Zhou]{Detang Zhou}


\newcommand{\M}{\mathcal M}

\begin{abstract} In this paper, we study self-expanders for mean curvature flows. First we show the discreteness of the spectrum of the drifted Laplacian on them. Next we give a universal lower bound of   the bottom  of the spectrum  of the drifted Laplacian and  prove that this lower bound is achieved if and only if the self-expander is the Euclidian subspace through the origin. Further, for self-expanders of codimension $1$, we prove  an inequality between  the bottom of the spectrum of the drifted Laplacian and the bottom of the spectrum of weighted stability operator and  that the hyperplane through the origin  is the unique self-expander where the equality holds. Also we prove the uniqueness of hyperplane through the origin for mean convex self-expanders under  some condition on the square of the norm of the second fundamental form.
\end{abstract}

\maketitle
\section{introduction}\label{introduction}

In this paper we study  self-expanding solutions  for the mean curvature flow (MCF), i.e. self-expanders. Recall that an $n$-dimensional smooth self-expander $\Sigma^n$  is a submanifold immersed in the Euclidean space $(\mathbb{R}^{n+k}, g_0)$, $k\geq 1$, satisfying 
\begin{equation}\label{intro-eq-1}
{\bf H}=\dfrac12x^{\perp},
\end{equation}
where $x$ denotes the position vector in  $\mathbb{R}^{n+k}$, $\perp$ denotes  the orthogonal projection onto the normal bundle of $\Sigma$,  and $ {\bf H}$ is the mean curvature vector of $\Sigma$ at $x$. 

In the case of codimension $1$, (\ref{intro-eq-1}) is equivalent to that the mean curvature $H$ satisfies
\begin{align}H=-\frac12\langle x, {\bf n}\rangle,
\end{align}
where ${\bf n}$ denotes the unit normal field of $\Sigma$.

Equivalently, $\Sigma$ is a self-expander if and only if $\sqrt{t}\Sigma, t\in (0,\infty)$ is a mean curvature flow.

  Self-expanders are very important in the study of MCF. They describe both the asymptotic  longtime behavior for MCF and the local structure of MCF after the singularities in the very short time.  See the works by Ecker and Huisken \cite{EH} and Stavrou \cite{S}.    Self-expanders  also arise  as the mean curvature evolution of  cones.  In Lecture 2 of \cite{I}, Ilmanen studied the existence of E-minimizing self-expanding hypersurfaces  which converge to prescribed closed cones at infinity in Euclidean space. It is known that the singular minimal cones are the singular self-expanders. Recently, Ding \cite{D} obtained some  results on minimal cones and self-expanders.  There are other works in self-expanders  (see, e.g. \cite{AIC},  \cite{FM}, \cite{S} and references therein).

  It is well known that  a self-expander $\Sigma$  is a critical point of the weighted volume functional with weighted volume element $e^{\frac{|x|^2}{4}}d\sigma$, where $d\sigma$ is the volume element of $\Sigma$. On the other hand it can be viewed as a minimal submanifold under the  conformal metric $e^{\frac{|x|^2}{2n}}g_0$ on $\mathbb{R}^{n+k}$  (see more details, e.g. \cite{CMZ3}, \cite{CMZ}).

On a self-expander, an important operator is the drifted Laplacian $\mathcal{L}=\Delta+\dfrac12\langle x,\nabla \cdot\rangle$. The reason is that $\mathcal{L}$ is a densely defined self-adjoint operator in the weighted $L^2$ space $L^2(\Sigma,  e^{\frac{|x|^2}{4}}d\sigma)$ of square integrable functions.

We discuss the spectral property of  the operator $\mathcal{L}$  and show the discreteness of  the  spectrum of $\mathcal{L}$. More precisely,
\begin{theorem} \label{th-1} For a complete $n$-dimensional properly immersed self-expander $\Sigma$ in $\mathbb{R}^{n+k}, k\geq 1$, the spectrum of the drifted Laplacian $\mathcal{L}=\Delta+\dfrac12\langle x,\nabla \cdot\rangle$ on $\Sigma$ is discrete. In particular, the bottom $\lambda_1$ of the spectrum of $\mathcal{L}$ is  the first  weighted $L^2$ eigenvalue of  $\mathcal{L}$.
\end{theorem}

The discreteness of the spectrum of $\mathcal{L}$ means that the embeding of the Sobolev space $W^{1,2}_0(\Sigma,  e^{\frac{|x|^2}{4}}d\sigma)$ into $L^2(\Sigma,  e^{\frac{|x|^2}{4}}d\sigma)$ is compact.  It is  worthy of  mentioning  that self-shrinkers for MCF also have the same property. In \cite{CZ13-2},  we showed that for a complete $n$-dimensional properly immersed self-shrinker  in $\mathbb{R}^{n+k}$, the  spectrum of  its  corresponding drifted Laplacian $\pazocal{L}=\Delta-\dfrac12\langle x,\nabla \cdot\rangle$ is discrete.  

For a complete properly immersed self-shrinker, $0$ is obviously the bottom of the spectrum of the operator $\pazocal{L}$ since the properness, polynomial volume growth, and the weighted volume are equivalent  for a self-shrinker  by the results  by Ding and Xin \cite{DX},   and  the authors \cite{CZ}.  For self-expanders, the situation is different. For instance, $\mathbb{R}^n$ has the infinite weighted volume.   We estimate the lower-bound of  the bottom $\lambda_1$ of the spectrum  and use the weighted $L^2$-integrability of the bottom $\lambda_1$ of the spectrum of $\mathcal{L}$ (Theorem \ref{th-1}) to discuss the  rigidity of    the lower bound.  More precisely, we prove that
\begin{theorem}\label{th-2} For a complete $n$-dimensional properly immersed self-expander $\Sigma$  in $\mathbb{R}^{n+k}, k\geq 1$,  the bottom of the spectrum of the drifted Laplacian $\mathcal{L}=\Delta+\dfrac12\langle x,\nabla \cdot\rangle$, i.e. the first weighted $L^2$ eigenvalue of $\mathcal{L}$, satisfies 
\begin{equation}\label{estimate-lambda-1}\lambda_1\geq \dfrac{n}{2}.
\end{equation}
The equality in (\ref{estimate-lambda-1}) holds if and only if $\Sigma$ is the Euclidian subspace $\mathbb{R}^n$ through the origin.
\end{theorem}

Further in the case of self-expander  hypersurfaces, using  Theorem \ref{th-1} and a Simons' type equation for the mean curvature $H$, we may obtain that
\begin{theorem} \label{th-3-hyper} If $\Sigma$ is a complete properly immersed self-expander hypersurface in $\mathbb{R}^{n+1}$, then the bottom $\lambda_1$ of the spectrum of the drifted Laplacian $\mathcal{L}=\Delta-\left<\nabla f, \nabla\cdot\right>$ on $\Sigma$, i.e. the first weighted $L^2$ eigenvalue of $\mathcal{L}$ satisfies
 \begin{equation}\label{3-eq-thm2-1-2}
 \lambda_1\geq \dfrac{n}{2}+\inf_{x\in \Sigma}H^2,
 \end{equation}
where $H$ is the mean curvature of $\Sigma$. The equality in (\ref{3-eq-thm2-1-2}) holds if and only if $\Sigma$ is the hyperplane $\mathbb{R}^n$ through the origin.
\end{theorem}

Observe that Theorem \ref{th-2} can be implied by Theorem \ref{th-3-hyper} when $\Sigma$ is a hypersurface.

A hypersurface is called mean convex if its mean curvature  is non-negative.
In \cite{CZ}, Colding and Minicozzi proved a complete embedded mean-convex self-shrinker hypersurface in $\mathbb{R}^{n+1}$ with polynomial volume growth must be cylinders $\mathbb{S}^{m}\times \mathbb{R}^{n-m}, 0\leq m\leq n$. For translating  solutions of MCF,  i.e. translators $H=-\langle w,{\bf n}\rangle$, where  $w$ is the constant vector in $\mathbb{R}^{n+1}$, in  \cite{TZ}, Tasayco and the second author in the present paper proved that an immersed nonflat translator $\Sigma^n$ in $ R^{n+1}, n = 2, 3$ is a grim hyperplane if and only if it is mean convex  and  there exists  a constant $C > 0$  such that
$\int_{B_R}|A|^2e^{\langle x, w\rangle}d\sigma\leq CR^2,
$
for all $R$ sufficiently large, where $B_R$ denotes the geodesic ball of $\Sigma$.  Their result is sharp in the sense that grim hyperplanes  for $n\geq 4$ do not satisfy the growth condition on the integral of $|A|$ given in the above. For self-expanders, it is known that there is no any compact one and  there are nonflat examples of mean-convex self-expanders (see, for instance, \cite{AIC}, \cite{EH}). Very recently, Fong and McGrath \cite{FM} showed that mean-convex self-expanders which are asymptotic to $O(n)$-invariant cones are rotationally symmetric. On mean convex self-expanders, by using Lemma 1 in \cite{TZ} and Simons' type equations,  we prove the following result:
\begin{theorem}\label{th-3} Let $\Sigma$ be a complete $n$-dimensional, $n\geq 2$,  properly immersed mean convex self-expander in $\mathbb{R}^{n+1}$  and let $B_R$ denote the geodesic ball of $\Sigma$  of radius $R$ centered in a fixed point in $\Sigma$.
 Suppose that  $h(t)$ is a positive function on $[\delta, \infty)$, for some $\delta>0$, satisfying $\frac t{h(t)}$ is nonincreasing on $[\delta,\infty)$ and 
\begin{align}\int_{\delta}^{\infty}\frac{t}{h(t)}dt<\infty.
\end{align}
If  there exists  a constant $C > 0$ such that the square $|A|^2$ of the norm of the second fundamental form of $\Sigma$  satisfies
\begin{equation}\label{eq-A-In}\int_{B_R}|A|^2e^{\frac{|x|^2}{4}}d\sigma<Ch(R),
\end{equation}
for all $R$ sufficiently large, 
 then $\Sigma$ must be a hyperplane $\mathbb{R}^n$ through the origin.

\end{theorem}

Note that some examples of $\kappa(t)$ in Theorem \ref{th-3} are $\kappa(t)= t^2$,  $t^2\log t$, $t^2(\log t)(\log\log t)$, and so on.
For instance, take $h(t)=t^2$, Theorem \ref{th-3}  implies the following
\begin{cor}\label{cor-mean} Let $\Sigma$ be a complete $n$-dimensional, $n\geq 2$,  properly immersed mean convex self-expanderin $\mathbb{R}^{n+1}$  and let $B_R$ denote the geodesic ball of $\Sigma$  of radius $R$ centered in a fixed point in $\Sigma$. If  there exists  a constant $C > 0$ such that the square $|A|^2$ of the norm of the second fundamental form of $\Sigma$  satisfies
\begin{equation}\label{eq-mean-1}\int_{B_R}|A|^2e^{\frac{|x|^2}{4}}d\sigma<CR^2,
\end{equation}
for all $R$ sufficiently large, 
 then $\Sigma$ must be a hyperplane $\mathbb{R}^n$ through the origin.
\end{cor}
With  the weight $e^{\frac{|x|^2}{4}}$,  inequality (\ref{eq-mean-1}) is a kind of asymptotic flatness condition. 

In the final part of this paper, we study  the $L$-stability operator for self-expanders, which  is the Sch$\ddot{\text{o}}$rdinger operator
\begin{equation} \label{L} L=\mathcal{L}+|A|^2-\frac12=\Delta+\frac12\left<x, \nabla\cdot\right>+|A|^2-\frac12.
\end{equation}

$L$-stability (see Section \ref{notation})  means that  the second variation of its weighted volume is nonnegative for any compactly supported normal variation (see, e.g., \cite{CMZ3}, \cite{D}). Recall that there is no  weighted-stable self-shrinkers with polynomial volume growth (\cite{CM1}).  Unlike self-shrinkers, the $L$-stability is not a rigid property for self-expanders. The self-expander graphs are weighted-volume minimizing and hence $L$-stable. Also the mean convex self-expanders are $L$--stable (see Theorem \ref{th-stab}). In spite of this, we may still study the bottom $\mu_1$ of the spectrum of  the operator $L$. Since a self-expander is noncompact, $\mu_1$ may  not be the lowest weighted $L^2$-eigenvalue for $L$ and also $\mu_1$ may take $-\infty$.
  We give an upper bound inequality for $\mu_1$ compared to the result of Colding-Minicozzi for self-shrinkers  (Theorem 9.2, \cite{CM}) .   More precisely, we obtain the following result:

\begin{theorem} \label{th-stab} Let $\Sigma$ be a complete $n$-dimensional properly immersed self-expander hypersurface. Then the bottom $\mu_1$ of the spectrum of the $L$-stability operator $L$ and the bottom  $\lambda_1$ of the spectrum of the drifted Laplacian $\mathcal{L}$ satisfy
\begin{align}\label{stab-eq-1}\mu_1\leq \lambda_1+\frac12.
\end{align}
The equality holds if and only if $\Sigma$ is a hyperplane through the origin.

Further, if $\Sigma$ is mean convex, then 
\begin{align}\label{stab-eq-2}1\leq \mu_1\leq \lambda_1+\frac12.
\end{align}

\end{theorem}

The rest of this paper is organized as follows: In Section 2 some definitions and notations are given,  in Section 3 we prove the results on the spectrum and bottom of the spectrum of the drifted Laplacian $\mathcal{L}$, in Section 4 we prove the rigidity theorem of the mean convex self-expanders and  finally in Section 5 we prove the inequality on the bottom of spectrum of $L$-stability operator $L$ and the bottom of spectrum of the drifted Laplacian $\mathcal{L}$. 

\section{Definitions and notation}\label{notation}

 For the convenience on some computations in this paper and simplicity of notations, we give the notation of $f$-minimal submanifolds (see, e.g. \cite{CMZ}).
 
Assume that    $\left(M,\bar{g}\right)$ is a smooth  $(n+k)$-dimensional  Riemannian manifold and $f$ is a smooth function on $M$. A smooth metric measure space is a triple $\left(M,\bar{g},e^{-f}dv\right)$ with  a weighted volume form $e^{-f}dv$ on $M$.
 Here $dv$ denote the volume element of $M$ induced by the  metric $\bar{g}$. 

Let $i: \Sigma^n\to (M^{n+k},\bar{g})$ denote the smooth   immersion of an $n$-dimensional submanifold $\Sigma$ into $M$.  Then $(M^{n+k},\bar{g})$ induces a metric, denoted by $g$ on $\Sigma$, such that $i$ is an isometric immersion.  Let  $d\sigma$ denote the volume element of $(\Sigma, g)$. Then the function $f$ restricted on $\Sigma$, still denoted by $f$,  induces a weighted volume element  $e^{-f}d\sigma$ on $\Sigma$ and thus   a smooth metric measure space $(\Sigma, g, e^{-f}d\sigma)$. 

The  isometric immersion $i: (\Sigma^n, g) \to (M^{n+k},\bar{g})$ is said to be properly immersed if,  for any compact subset $\Omega$ in $M$, the pre-image $i^{-1}(\Omega)$ is compact in $\Sigma$.

In this paper, unless otherwise specified, the notations with a bar denote the quantities
corresponding the metric $\bar{g}$ on $M$. For instance $\bar{\nabla}$ and  $\bar{\nabla}^2$ denote  the Levi-Civita connection, and Hession  of $(M, \bar{g})$ respectively. On the other hand, the notations  like ${\nabla}$ denote the quantities corresponding the intrinsic metric ${g}$ on $\Sigma$.
 For instance $\nabla, \Delta$  denote the Levi-Civita connection and  the Laplacian on $(\Sigma, g)$ respectively.

 Let $A$ denote   the second fundamental form of $\Sigma$. The mean curvature vector  ${\bf H}(p)$ of $\Sigma$ at $p\in \Sigma $ is defined  by 
\[{\bf H}(p):=\sum_{i=1}^n(\bar\nabla_{e_i}e_i)^\perp=\sum_{i=1}^nA(e_i,e_i),
\]
where $\{ e_1, e_2, \cdots, e_n\}$ is a local orthonormal frame of  $\Sigma$ at $p$ and    $\perp$ denotes the projection onto the normal bundle of $\Sigma$.

The weighted volume of a measurable subset  $S\subset \Sigma$ is defined by
\begin{equation}\label{notation-eq-vol}V_f(S):=\int_S e^{-f}d\sigma.
\end{equation}

\begin{definition} \label{def-f-min} The weighted mean curvature vector ${\bf H}_f$  of the submanifold $(\Sigma,{g})$  is defined  by 
\begin{equation}{\bf H}_f:={\bf H}+(\overline{\nabla }f)^{\perp}. 
\end{equation}

A submanifold  $(\Sigma, g)$  is called $f$-minimal  if  its weighted mean curvature ${\bf H}_f$ vanishes identically, or equivalently if it satisfies
\begin{equation} \label{f-min}{\bf H}=-(\overline{\nabla} f)^\perp.
\end{equation}

In the case of hypersurfaces, the mean curvature of $\Sigma$  is defined by 
\begin{equation}{\bf H}=-H{\bf n},
\end{equation}
where ${\bf n}$ is the unit normal field on $\Sigma$.  
\end{definition}
From Definition  \ref{def-f-min}, the hypersurface $\Sigma$ is $f$-minimal if and only if 
\[H_f=H-\left<\bar{\nabla}f, {\bf n}\right>=0,
\]
or equivalently
\begin{equation}
H=\left<\bar{\nabla}f, {\bf n}\right>.
\end{equation}

It is known that  an $f$-minimal submanifold  is a critical point of the weighted volume functional defined in (\ref{notation-eq-vol}). On the other hand,  it is also  a minimal submanifold under the conformal metric $\tilde{g}=e^{-\frac{2}{n}f}\bar{g}$ on $M$  (see, e.g. \cite{CMZ3}, \cite{CMZ}).

When $(M, \bar{g})$ is the Euclidean space $(\mathbb{R}^{n+k}, g_0)$, we have some interesting examples of $f$-minimal submanifolds:
\begin{example}\label{example1} If $f=\frac{|x|^2}4$,  $-\frac{|x|^2}4$, and  $-\left<x,w\right>$ respectively,  where $w\in \mathbb{R}^{n+k}$ is a constant vector, an $n$-dimensional $f$-minimal submanifold $\Sigma$ is a  self-shrinker, self-expander and translator for MCF in the Euclidian space $\mathbb{R}^{n+k}$ respectively. 
\end{example}

For  $(\Sigma, g, e^{-f}d\sigma)$, let $L^2(\Sigma, e^{-f}d\sigma)$ denote the space of square-integrable functions on $\Sigma$ (not necessarily a submanifold) with respect to the measure $e^{-f}d\sigma$.
The weighted $L^{2}$ inner product of functions $u$ and $v$ in  $L^2(\Sigma, e^{-f}d\sigma)$ is defined by 
\[
\left\langle u,v\right\rangle _{L^2(\Sigma, e^{-f}d\sigma)}=\int_{\Sigma}uve^{-f}d\sigma.
\]

The drifted Laplacian on $\Sigma$ is defined by 
\[
\Delta_{f}=\Delta-\left\langle \nabla f,\nabla\cdot\right\rangle .
\]

It is well known that $\Delta_f$ is a densely defined self-adjoint
operator in $L^2(\Sigma, e^{-f}d\sigma)$, i.e. for 
$u$ and $v$ in $C^{\infty}_0(\Sigma)$, it holds that 
\begin{equation}
\int_{\Sigma}(\Delta_{f}u) ve^{-f}d\sigma=-\int_{\Sigma}\left\langle \nabla u,\nabla v\right\rangle e^{-f}d\sigma.
\end{equation}

We recall some facts in spectral theory (see more details in,  e.g. \cite{G},  \cite{RS}). Consider the Schr$\ddot{\text{o}}$dinger operator on $\Sigma$:
$$S=\Delta_f+q, \quad q\in L^{\infty}_{loc}(\Sigma).$$
The  weighted $L^2$ spectrum of $S$  is called the spectrum of $S$ for short whenever there is no confusion.
The bottom $s_1$ of the spectrum of $S$ can be characterized by
\begin{align}\label{bottom-0}
\displaystyle s_1=\inf\left\{\frac{\int_{\Sigma}\left(|\nabla\varphi|^{2}-q\varphi^2\right)e^{-f}d\sigma}{\int_{\Sigma}\varphi^2e^{-f}d\sigma};  \varphi\in C_0^{\infty}(\Sigma),  \int_{\Sigma}\varphi^2e^{-f}d\sigma\neq 0\right\}.
\end{align}
A number $s$ is said to be a weighted $L^2$ eigenvalue to $S$ if there exists a smooth nonzero function $u\in L^2(\Sigma, e^{-f}d\sigma)$ satisfying
\begin{equation}\label{eq-schro-1}Su+s u=0,
\end{equation}
The function $u$ in (\ref{eq-schro-1}) is called   the weighted $L^2$ eigenfunction associated to $s$.

In general, if $\Sigma$ is non-compact, the bottom $s_1$ may not be  a weighted $L^2$ eigenvalue and may be $-\infty$. The spectral theory says that if  the spectrum of $S$ is discrete,  the spectrum is the set of all the weighted $L^2$ eigenvalues of $S$, counted with multiplicity, which   is an increasing sequence 
\[ s_1< s_2\leq s_3\leq\cdots
\] with $s_i\to \infty$ as $i\to \infty$.  Further, the variational characterization of $s_i$ states that  the bottom $s_1$ of spectrum  of $S$ is just  the first weighted $L^2$ eigenvalue  with multiplicity $1$.

 Now we  give special notations for self-expanders.  In the following, unless otherwise specified, let $\Sigma$ be an $n$-dimensional  self-expander in $\mathbb{R}^{n+k}, k\geq 1$, that is, $\Sigma$ satisfies the equation
\begin{align}{\bf H}=\dfrac{x^{\perp}}{2}
\end{align}

In the case of codimension $1$, $\Sigma$ is a self-expander if and only if the mean curvature $H$ satisfies that
\begin{align}H=-\frac12\langle x, {\bf n}\rangle.
\end{align}

Observe that  a self-expander $\Sigma$ can be viewed as an $f$-minimal submanifold in $\mathbb{R}^{n+k}$ by taking  $f=-\frac{|x|^2}{4}$ in equation (\ref{f-min}).

The weighted volume of a measurable subset  $S\subset \Sigma$ is given by
\begin{equation}\label{notation-expand-vol}V_f(S):=\int_S e^{\frac{|x|^2}{4}}d\sigma.
\end{equation}

The weighted $L^{2}$ inner product of functions $u$ and $v$ in  $L^2(\Sigma, e^{\frac{|x|^2}{4}}d\sigma)$ is defined by 
\[
\left\langle u,v\right\rangle _{L^2(\Sigma, e^{\frac{|x|^2}{4}}d\sigma)}=\int_{\Sigma}uve^{\frac{|x|^2}{4}}d\sigma.
\]

Denote by  $\mathcal{L}$ the drifted Laplacian on $\Sigma$, i.e.  $\mathcal{L}=\Delta+\frac12\left<x,\nabla\cdot\right>$.

The bottom $\lambda_1$ of the spectrum of $\mathcal{L}$  can be given by
\begin{align}\label{bottom-1}
\displaystyle\lambda_1=\inf\left\{\frac{\int_{\Sigma}|\nabla \varphi|^2e^{\frac{|x|^2}{4}}d\sigma}{\int_{\Sigma}\varphi^2e^{\frac{|x|^2}{4}}d\sigma};  \varphi\in C_0^{\infty}(\Sigma),  \int_{\Sigma}\varphi^2e^{\frac{|x|^2}{4}}d\sigma\neq 0\right\}.
\end{align}

From (\ref{bottom-1}), $\lambda_1$ is nonnegative.
 For self-expanders,   the stability operator, which appears in the second variation of the weighted volume, is a   Schr$\ddot{\text{o}}$dinger operator given by
\begin{equation} \label{L} L=\mathcal{L}+|A|^2-\frac12=\Delta+\frac12\left<x, \nabla\cdot\right>+|A|^2-\frac12.
\end{equation}

\begin{definition}A self-expander $\Sigma$ is said to be $L$-stable  if the following inequality holds  for all  $\varphi\in\mathit{C}^{\infty}_{0}(\Sigma)$,
\begin{equation}\label{f-stable-ine}
-\int_{\Sigma}\varphi (L\varphi)e^{\frac{|x|^2}{4}}d\sigma=\int_{\Sigma}\left(|\nabla\varphi|^{2}-\bigr(|A|^{2}-\frac12)\varphi^2\right)e^{\frac{|x|^2}{4}}d\sigma\geq 0.
\end{equation}
\end{definition}

 $L$-stability  of $\Sigma$ is equivalent to that   the second variation of its weighted volume is nonnegative for any compactly supported normal variation. Denote the bottom of the spectrum of $L$ by $\mu_1$.
$L$-stability  means  $\mu_1\geq 0$.

\section{Spectral properties of the drifted Laplacian}\label{spectrum}

In this section, we show   the
discreteness of the spectrum of the drifted Laplacian  $\mathcal{L}=\Delta+\frac12\left<x, \nabla\cdot\right>$ and study  the bottom of the  spectrum of $\mathcal{L}$ for complete properly immersed self-expanders.  
We start by giving the following identities:
\begin{lemma}\label{lem} For a manifold $(\Sigma, g)$ (not necessarily submanifold), 
\begin{equation}\label{eq-6}
\Delta_f(ue^h)=e^h\left\{\Delta_{f-2h}u+[\Delta h+\langle\nabla (h-f),\nabla h\rangle]u\right\},
\end{equation}
where $f$ and $h$ are smooth functions on $\Sigma$.

In particular, if $h=f$,
\begin{equation}\label{3-eq-lem-1}
\Delta_f(ue^f)=e^f\left[\Delta_{(-f)}u+(\Delta f) u\right].
\end{equation}
\end{lemma}
\begin{proof}The results come from the direct computations. 
\begin{align}
\nabla(ue^h)&=e^h(\nabla u+u\nabla h)\label{eq-grad}\\
\Delta (ue^h)&=e^h\left[\Delta u+2\langle\nabla u,\nabla h\rangle+u(\Delta h+|\nabla h|^2)\right]\label{eq-lap}
\end{align}
By  (\ref{eq-grad}) and (\ref{eq-lap}), we have
\begin{align}\label{eq-weigh}
\Delta_f (ue^h)&=e^h\left[\Delta u+2\langle\nabla u,\nabla h\rangle+u(\Delta h+|\nabla h|^2)\right]\nonumber\\
&\qquad -e^h(\langle\nabla f,\nabla u\rangle+u\langle\nabla f,\nabla h\rangle)\nonumber\\
&=e^h[\Delta u-\langle\nabla(f-2 h),\nabla u\rangle+u(\Delta h+\langle\nabla h,\nabla (h-f)\rangle]
\end{align}
\end{proof}
Recall that a function $h: \Sigma\rightarrow \mathbb{R}$  is said to be  proper if, for any bounded closed subset  $I\subset \mathbb{R}$, the inverse image $h^{-1}(I)$ is compact in $\Sigma$. Now we prove Theorem \ref{th-1}. 
\medskip

 {\it Proof of Theorem \ref{th-1}.} Take $f=-{\frac{|x|^2}{4}}$. 
Consider the unitary isomorphism 
\[U:L^{2}\left(\Sigma, d\sigma\right)\to L^2\left(\Sigma, e^{-f}d\sigma\right)=L^{2}(\Sigma, e^{\frac{|x|^2}{4}}d\sigma)
\]
given by $Uu=ue^{\frac{f}{2}}$. Take $h=\dfrac{f}{2}$ in \eqref{eq-6} and note $\mathcal{L}=\Delta_{f}$. 
  We have
\[
\mathcal{L}=UTU^{-1},  \quad T=\Delta+\frac{1}{2}\Delta f-\frac{1}{4}\left|\nabla f\right|^{2}.
\]

By the spectral theory, the discreteness  of the spectrum of $\mathcal{L}$  and the discreteness of the spectrum of $T$ are equivalent.

Note that ${\bf H}=-(\bar{\nabla}f)^{\perp}$ and $f=-{\frac{|x|^2}{4}}$.  It holds that 
\begin{align}\label{3-eq-prop-7}
\Delta f&=\sum_{i=1}^n\overline{\nabla}^2f(e_i,e_i)+\left\langle \overline{\nabla}f,{\bf H}\right\rangle=-\dfrac{n}{2}-|\bf{H}|^2.
\end{align}
Since Sigma is a self-expander, we have
\begin{align}\label{3-eq-prop-1}
\frac{1}{4}\left|\nabla f\right|^{2}-\frac{1}{2}\Delta f  
 =&\frac{1}{4}\left|\overline{\nabla}f\right|^{2}+\frac{1}{4}|\left(\overline{\nabla}f\right)^{\perp}|^{2}+\frac n4\nonumber\\
 =&\frac14|x|^2+\frac1{16}|x^{\perp}|^2+\frac n4\nonumber\\
 \geq &\frac14|x|^2+\frac n4
\end{align}

Since $\Sigma$ is properly immersed in $\mathbb{R}^{n+k}$,   the function $\frac14|x|^2+\frac n4$ restricted on $\Sigma$ is proper and tends to $\infty$ when the intrinsic distance $d^{\Sigma}(p,x)$ of $\Sigma$  tends $ \infty$, where $p\in \Sigma$ fixed. (\ref{3-eq-prop-1}) implies that $\frac{1}{4}\left|\nabla f\right|^{2}-\frac{1}{2}\Delta f $  also tends to $\infty$ when  $d^{\Sigma}(p,x)\rightarrow \infty$. This implies that  the spectrum of  the operator $T$ is discrete  (e.g. \cite{RS} page 120). Hence, the spectrum of $\mathcal{L}$ is discrete.

By the discreteness of the spectrum of the drifted Laplacian $\mathcal{L}$, the variational character of  the first  weighted $L^2$ eigenvalue of  $\mathcal{L}$  implies that it is just the bottom $\lambda_1$ of the spectrum of $\mathcal{L}$.

\qed

We consider the spectrum of $\mathcal{L}$  on $\mathbb{R}^n$ and obtain  that
\begin{prop}   i) The spectrum of $\mathcal{L}=\Delta+\frac12\left<x, \nabla\cdot\right>$  on $\mathbb{R}^n$ is discrete and 
 the weighted $L^2$ eigenvalues of $\mathcal{L}$,  counted by multiplicity,  are the following 
\begin{equation}
\lambda=\frac n2+\frac12\sum_{i=1}^nk_i, \quad k_i\in \{0\}\cup\mathbb{N},
\end{equation}
where  $\frac{k_i}{2}, $  are the weighted $L^2$ eigenvalues of the operator $\frac{d^2}{dx_i^2}-\frac {x_i}2\frac{d}{dx_i}, x_i\in\mathbb{R}$ with respect to the weighted measure $e^{-\frac{x_i^2}{4}}dx_i$, that is, the weighted $L^2$ eigenfunctions associated to $\frac{k_i}{2}$ are the Hermite polynomials $H_{k_i}(\frac{x_i}{2})$.

ii) The weighted $L^2$ eigenfunction associated to $\lambda$, counted by multiplicity,  is $\Phi=\Pi_{i=1}^n{H}_{k_i}(\frac{x_i}{2})e^{-\frac{|x|^2}{4}}$, where $\Pi_{i=1}^n{H}_{k_i}(\frac{x_i}{2})$ are the products of  the corresponding Hermite polynomials.

iii) The bottom $\lambda_1$ of the spectrum of $\mathcal{L}$, that is,   the first weighted $L^2$ eigenvalue of $\mathcal{L}$ is $\frac n2$ with multiplicity $1$;
\end{prop}
\begin{proof} Take $f=-\frac{|x|^2}{4}$ in  (\ref{3-eq-lem-1}). Note, $\mathcal{L}=\Delta+\frac12\left<x, \nabla \cdot\right>$ and $\pazocal{L}=\Delta-\frac12\left<x, \nabla \cdot\right>$.
\begin{align}\label{R-1}\mathcal{L}(ue^{-\frac{|x|^2}{4}})=e^{-\frac{|x|^2}{4}}(\pazocal{L}-\frac n2)u
\end{align}
Observe that $u\in L^2(\mathbb{R}^n, e^{-\frac{|x|^2}{4}}d\sigma)$ if and only if $ue^{-\frac{|x|^2}{4}}\in L^2(\mathbb{R}^n, e^{\frac{|x|^2}{4}}d\sigma)$.
\begin{align}\label{spec-L}
\pazocal{L}&=\sum_{i=1}^n(\dfrac{\partial^2}{\partial x_i^2}-\frac{x_i}2\frac{\partial}{\partial x_i})
\end{align}
For the operator $\frac{d^2}{dt^2}-\frac{t}{2}\frac{d}{dt}, t\in \mathbb{R}$, it is known that its spectrum  on $L^2(\mathbb{R},e^{-\frac{t^2}{4}}dt)$ is discrete and the Hermite polynomials $H_{k}(\frac x2)$ are  orthonormal eigenfunctions associated to $\frac{k}{2}$, $k\in \{0\}\cup\mathbb{N}$,  which form a complete orthonormal system for space $L^2(\mathbb{R}, e^{-\frac{t^2}{4}}dt)$.  
By this fact and (\ref{spec-L}),  one can verify  that the products  $\Pi_{i=1}^n{H}_{k_i}(\frac{x_i}{2})$ of  $n$  eigenfunctions  $H_{k_i}(\frac{x_i}{2})$ of $\frac{d^2}{dx_i^2}-\frac{x_i}{2}\frac{d }{dx_i}$ respectively, $1\leq i\leq n, k_i\in \{0\}\cup \mathbb{N}$,  are the eigenfunctions of $\pazocal{L}$ associated to the eigenvalue $\displaystyle\frac12\sum_{i=1}^nk_i$ and, by a standard argument in functional analysis,,  form a complete system for  the space $L^2(\mathbb{R}^n, e^{-\frac{|x|^2}{4}}d\sigma)$. 

By (\ref{R-1}), $\mathcal{L}$ has the discrete spectrum and the weighted $L^2$ eigenvalues,  counted by multiplicity,  are the following 
\begin{equation}
\lambda=\frac n2+\frac12\sum_{i=1}^nk_i, \quad k_i\in \{0\}\cup\mathbb{N},
\end{equation}
and associated eigenfunctions are $\Phi=\Pi_{i=1}^n{H}_{k_i}(\frac{x_i}{2})e^{-\frac{|x|^2}{4}}$, which form a complete orthonormal system 
for  the space $L^2(\mathbb{R}^n, e^{\frac{|x|^2}{4}}d\sigma)$.

\end{proof}

Now we prove Theorem \ref{th-2} in which we study the universal lower bound for the  bottom of the spectrum of $\mathcal{L}$ for self-expanders.
\medskip

{\it Proof of Theorem \ref{th-2}.}  Taking $u\equiv 1$  and $f=-\frac{|x|^2}{4}$ in (\ref{3-eq-lem-1}) yields 
\begin{equation*}
\mathcal{L}(e^{-\frac{|x|^2}{4}})=e^{-\frac{|x|^2}{4}}\left[(\Delta (-\frac{|x|^2}{4}) \right].
\end{equation*}
 Using (\ref{3-eq-prop-7}),  we have the positive function $v=e^{-\frac{|x|^2}{4}}>0$ satisfying that
\begin{align}\label{3-eq-expan-0}
\mathcal{L}v+(\dfrac{n}{2}+|{\bf H}|^2)v=0.
\end{align}

It is well known that  (\ref{3-eq-expan-0}) implies that for any  $\varphi\in C_0^{\infty}(\Sigma)$,
\begin{align}\int_{\Sigma}\left[|\nabla \varphi|^2-(\frac n2+|{\bf H}|^2)\varphi^2\right]e^{\frac{|x|^2}{4}}d\sigma\geq 0.
\end{align}

Then the bottom $\lambda_1$ of spectrum of $\mathcal{L}$ satisfies that 
\begin{align}\label{ine-lambda1-1}
\lambda_1= \inf_{\varphi\in C_0^{\infty}(\Sigma),\int \varphi^2e^{\frac{|x|^2}{4}}d\sigma\neq 0}\dfrac{\int|\nabla \varphi|^2e^{\frac{|x|^2}{4}}d\sigma}{\int \varphi^2e^{\frac{|x|^2}{4}}d\sigma}\geq \dfrac{n}{2}+\inf_{x\in \Sigma}|{\bf H}||^2\geq \frac n2.
\end{align}

So we have proved that the inequality holds. 

Now we assume that   the equality $\displaystyle\lambda_1=\frac n2$ holds.  By Theorem \ref{th-1},   the spectrum of $\mathcal{L}$ of $\Sigma$  is discrete and $\lambda_1$ is the first  weighted $L^2$ eigenvalue of $\mathcal{L}$.  Then there exists the first eigenfunction $u>0$  such that $u\in W^{1,2}(\Sigma)$ and
 \begin{align}\label{3-eq-expan-1-1}\mathcal{L}u+\lambda_1u=0.
 \end{align}

Note $v=e^{-\frac{|x|^2}{4}}>0$. By (\ref{3-eq-expan-0}),
\begin{align}
\mathcal{L}\log v&=\dfrac{\mathcal{L}v}{v}-\dfrac{|\nabla v|^2}{v^2}=-\dfrac n2-|{\bf H}|^2-|\nabla \log v|^2
\end{align}
Take $\phi\in \mathit{C}_0^{\infty}(\Sigma)$.
\begin{align}\label{eq-bottom-1}
\int_{\Sigma}\left(\dfrac{n}{2}+|{\bf H}|^2+|\nabla \log v|^2\right)\phi^2e^{\frac{|x|^2}{4}}d\sigma&=-\int_{\Sigma}\phi^2(\mathcal{L}\log v)e^{\frac{|x|^2}{4}}d\sigma\nonumber\\
&=\int_{\Sigma}\left<\nabla \phi^2, \nabla \log v\right>e^{\frac{|x|^2}{4}}d\sigma
\end{align}
Since
\[\int_{\Sigma}\left<\nabla \phi^2, \nabla \log w\right>e^{\frac{|x|^2}{4}}d\sigma\leq \int_{\Sigma}|\nabla \phi|^2e^{\frac{|x|^2}{4}}d\sigma
+\int_{\Sigma}|\nabla \log w|^2\phi^2e^{\frac{|x|^2}{4}}d\sigma,
\]
(\ref{eq-bottom-1}) implies that
\begin{align}\label{eq-bottom-2}
\int_{\Sigma}(\dfrac{n}{2}+|{\bf H}|^2)\phi^2e^{\frac{|x|^2}{4}}d\sigma
&\leq\int_{\Sigma}|\nabla \phi|^2e^{\frac{|x|^2}{4}}d\sigma.
\end{align}

Choose $\phi=\varphi_j u$, where $\varphi_j$ are the non-negative cut-off functions   satisfying  that  $\varphi_j$ is $1$ on $B_j$,  $|\nabla\varphi|\leq 1$ on $B_{j+1}\setminus B_j$, and $\varphi=0$ on $\Sigma\setminus B_{j+1}$.  Substitute $\phi$ in (\ref{eq-bottom-2}):
\begin{align}\label{eq-bottom-3}
\int_{\Sigma}|(\dfrac{n}{2}+|{\bf H}|^2)\varphi_i^2u^2e^{\frac{|x|^2}{4}}d\sigma
&\leq\int_{\Sigma}|\nabla (\varphi_ju)|^2e^{\frac{|x|^2}{4}}d\sigma.
\end{align}
Note that
\begin{align}\label{eq-bottom-4}
\int_{\Sigma}|\nabla (\varphi_ju)|^2e^{\frac{|x|^2}{4}}d\sigma&\leq 2\int_{\Sigma}\varphi_j^2|\nabla u|^2e^{\frac{|x|^2}{4}}d\sigma+2\int_{\Sigma}|\nabla\varphi_j|^2u^2e^{\frac{|x|^2}{4}}d\sigma,
\end{align}
and $u\in W^{1,2}(\Sigma)$.  Letting $j\rightarrow \infty$ in  (\ref{eq-bottom-4}) and using the monotone convergence theorem, 
\begin{align}\label{eq-bottom-5}
\int_{\Sigma}|\nabla (\varphi_ju)|^2e^{\frac{|x|^2}{4}}d\sigma&\rightarrow \int_{\Sigma}|\nabla u|^2e^{\frac{|x|^2}{4}}d\sigma.
\end{align}
 Letting $j\rightarrow \infty$ in  (\ref{eq-bottom-3}) and using the monotone convergence theorem again, we have
\begin{align}\label{eq-bottom-6}
\int_{\Sigma}(\dfrac{n}{2}+|{\bf H}|^2)u^2e^{\frac{|x|^2}{4}}d\sigma
&\leq \int_{\Sigma}|\nabla u|^2e^{\frac{|x|^2}{4}}d\sigma\nonumber\\
&=\frac n2\int_{\Sigma}u^2e^{\frac{|x|^2}{4}}d\sigma.
\end{align}

Hence $|\bf H|=0.$ Since ${\bf H}= \frac12 x^{\perp}=0$. $\Sigma$ is an $n$-dimensional  minimal cone. Since $\Sigma$ is smooth, $\Sigma$ must be $\mathbb{R}^n$.  By the self-expander equation, it passes through the  origin.

\qed

It is known that a submanifold in  $\mathbb{R}^{n+k}, k>0,$ is minimal if and only if the coordinate functions restricted on $\Sigma$ are the harmonic functions. For self-expanders, there are similar properties:

\begin{prop} An $n$-dimensional  immersed submanifold  $\Sigma$ in $\mathbb{R}^{n+k}, k>0,$ is a self-expander if and only if the coordinate functions restricted on $\Sigma$ are eigenfunctions of the drifted Laplacian $\mathcal{L}$. Moreover, they are corresponding the eigenvalue $-\frac12$, i.e. they satisfy
\[
\mathcal{L}x_i-\frac 12x_i=0
\]
for all $i=1,2,\ldots, n+k$.
\end{prop}
\begin{proof}   Take $f=-\frac{|x|^2}{4}.$
\begin{align*}
\mathcal{L}x_i&=\Delta x_i-\left<\nabla f, \nabla x_i\right>\\
&=\sum_{j=1}^n\bar{\nabla}^2x_i(e_j,e_j)+\left<{\bf H}, \bar{\nabla} x_i\right>-\left<\nabla f, \nabla x_i\right>\\
&=\left<{\bf H}+(\bar{\nabla}f)^{\perp}, \bar{\nabla} x_i\right>-\left<\bar{\nabla}f, \bar{\nabla}x_i\right>\\
&=\left<{\bf H}_f, \bar{\nabla} x_i\right>+\dfrac{x_i}{2}.\\
\end{align*}
In the above $\{ e_1, e_2, \cdots, e_n\}$ is an orthonormal basis of the tangent space $T_p\Sigma$, $p\in \Sigma$.
\begin{equation}\mathcal{L}x_i-\dfrac{x_i}{2}=\left<{\bf H}_f, \bar{\nabla} x_i\right>.
\end{equation}
So the left-hand side is zero for all $i$ if and only only if $\langle {\bf H}_f,e_i \rangle=0$ for $i$ which is equivalent to ${\bf H}_f=0$. 

\end{proof}

\begin{remark}The above results hold locally. For properly immersed submanifolds, it has been  known that the coordinate functions are the weighted $L^2$ eigenfunctions  for the corresponding drifted Laplacian for a self-shrinker. For self-expanders, the corresponding drifted Laplacian still has discrete spectrum but the coordinate functions are not the weighted  $L^2$ eigenfunctions. 
\end{remark}

In the rest of this section,  we discuss self-expanders of codimension $1$  whose  equation is
\[ H=-\frac12\langle x, {\bf n}\rangle.
\]
Recall that in \cite{CMZ2}, Mejia and the authors calculated the Simons' type equations for general $f$-minimal hypersurfaces as follows.
\begin{prop}\label{prop1} (Propostion 1 in \cite{CMZ2})
Let $(\Sigma^{n}, g)$ be an $f$-minimal hypersurface isometrically immersed in a smooth metric measure space $(M,\bar{g},e^{-f}d\mu)$. Then the  mean curvature $H$ of $\Sigma$ satisfies that 
\begin{eqnarray}\label{f-lap-H-1}
\Delta_fH&=&2\sum_{i=1}^{n}(\overline{\nabla}^3f)_{i\nu i}-\sum_{i=1}^n(\overline{\nabla}^3f)_{\nu i i}+2\sum_{i,j=1}^{n}a_{ij}(\overline{\nabla}^2f)_{ij}\\
&&-\overline{\textrm{Ric}}_{f}(\nu,\nu)H-|A|^{2}H,\nonumber
\end{eqnarray}
where $\{e_1,\ldots, e_n\}$ is a local orthonormal frame field on $\Sigma$, $\nu$ denotes the unit normal to $\Sigma$, $a_{ij}=A(e_i,e_j)$, and $\overline{\textrm{Ric}}_{f}=\overline{\textrm{Ric}}+\overline{\nabla}^2f$.
\end{prop}

\begin{prop} \label{simons-eq-grad}(Corollary  3 in \cite{CMZ2}) Let $(M^{n+1},\overline{g}, e^{-f}d\mu)$ be a smooth metric metric space satisfying $\overline{Ric}_f=C\overline{g}$, where $C$ is a constant. If $(\Sigma,g)$ is an $f$-minimal hypersurface isometrically immersed in $M$, then it holds that on $\Sigma$
\begin{eqnarray}
\frac{1}{2}\Delta_f|A|^{2}&=&|\nabla A|^{2}+C|A|^{2}-|A|^{4}
+\sum_{i,j=1}^{n}a_{ij}\bar{R}_{i\nu j\nu;\nu}\\
& &-2\sum_{i,j,k=1}^{n}a_{ij}a_{ik}\bar{R}_{j\nu k\nu}-2\sum_{i,j,k,l=1}^{n}a_{ij}a_{lk}\bar{R}_{iljk},\nonumber
\end{eqnarray}
where the notation is the same as in Propostion \ref{prop1}.
\end{prop}

Take $f=-\frac{|x|^2}{4}$ and $M=\mathbb{R}^{n+k}$ and note $H=\displaystyle\sum_{i=1}^na_{ii}$, $\displaystyle|A|^2=\sum_{i=1}^na_{ii} ^2$. It holds that
the Simons' type equations for self-expanders:
\begin{lemma} For self-expanders of codimension $1$,
\begin{align}\label{simons-H}
\mathcal{L}H+(|A|^2+\frac{1}{2})H=0.
\end{align}
and 
\begin{align}\label{simons-A}
\dfrac12\mathcal{L}|A|^2=|\nabla A|^2-\dfrac12|A|^2-|A|^4.
\end{align}
\end{lemma}
  (See (\ref{simons-H}) also  in \cite{D}). Using (\ref{simons-H}),  we can give
\medskip

 {\it Proof of Theorem  \ref{th-3-hyper}.}  Similar to  the proof of Theorem \ref{th-2},  we have   (\ref{ine-lambda1-1}) gives 
 \begin{align}\label{ine-lambda1-1-1}
\lambda_1\geq \dfrac{n}{2}+\inf_{x\in \Sigma}H^2.
\end{align}

Suppose that  the equality $\displaystyle\lambda_1=\frac n2+\inf_{x\in \Sigma}H^2$ holds.   Again there exists the first eigenfunction $u>0$  such that $u\in W^{1,2}(\Sigma)$ and
 \begin{align}\label{3-eq-expan-1}\mathcal{L}u+\lambda_1u=0.
 \end{align}
 The proof of Theorem \ref{th-2} gives
\begin{align}\label{ine-lambda1-2}
\int_{\Sigma}(\dfrac{n}{2}+H^2)u^2e^{\frac{|x|^2}{4}}d\sigma
&\leq \int_{\Sigma}|\nabla u|^2e^{\frac{|x|^2}{4}}d\sigma=\lambda_1\int_{\Sigma}u^2e^{\frac{|x|^2}{4}}d\sigma.
\end{align}

This implies that $\displaystyle H^2=\inf_{x\in \Sigma}H^2=C.$  
The Simons' type equation (\ref{simons-H}) says that
\begin{align}
\mathcal{L}H+(|A|^2+\frac{1}{2})H=0.\nonumber
\end{align}

It implies that $ H=0$.  Hence $\Sigma$ must be $\mathbb{R}^n$ through the origin.

\qed

\section{Mean convex self-expanders of codimension $1$}

A self-expander hypersurface $\Sigma$ is called mean convex if its mean curvature $H\geq 0$. Besides the hyperplane $\mathbb{R}^n$ through the origin, there are nontrivial examples (see, e.g. \cite{EH}). In this section,
We will prove some rigidity results on mean convex self-expander hypersurfaces.

First we need  the following result proved byTasayco and the second author in   \cite{TZ}:

\begin{lemma} \label{lem1} (Lemma 1 in  \cite{TZ})
On a complete weighted manifold $\displaystyle \left(M,\left\langle ,\right\rangle,e^{-f}d\textsc{\footnotesize vol}\right)$, assume that the functions $u$, $v \in C^2\left(M\right)$, with $u > 0$ and $v\ge 0$ on $M$, satisfy
\begin{eqnarray}
\displaystyle {\Delta}_f u + q\left(x\right) u \le 0 \qquad \qquad \mbox{and} \qquad \qquad {\Delta}_f v + q\left(x\right) v \ge 0, \label{eq11} 
\end{eqnarray}
where $q\left(x\right) \in C^{0}\left(M\right)$. Suppose that there exists a positive  function $\kappa > 0$  on $[\delta, \infty)$ for some $\delta>0$, satisfying  $\frac{t}{\kappa(t)}$ is nonincreasing on $[\delta,\infty)$ and 
\begin{align}\int_{\delta}^{+\infty}\frac{t}{\kappa(t)}dt=+\infty,
\end{align}such that
\begin{equation}
\displaystyle \int_{B_R}v^2 e^{-f} \ \leq \kappa(R) \label{eq12}	
\end{equation} 
for all $R$ sufficiently large, where $B_R$ denotes the geodesic ball of radius $R$ of $\Sigma$. Then there exists a constant $C$ such that $\displaystyle v = Cu$.
\end{lemma}

Using Lemma \ref{lem1},  we can show Theorem  \ref{th-3}.

\medskip

 {\it Proof of Theorem  \ref{th-3}.} Since $H\geq 0$ and $\mathcal{L}H+(|A|^2+\dfrac12)H=0,$ by the argument using the Harnack inequality we have $H\equiv 0$ or $H>0$ on $\Sigma$. If $H\equiv 0$, it is hyperplane $\mathbb{R}^n$ through the origin. Assume  that $H>0$ on $\Sigma$.
It holds that
\begin{align}\label{simons-H-1}\mathcal{L}H+(|A|^2+\dfrac12)H=0.
\end{align}
We also have the Simons' type equation (\ref{simons-A}) for $|A|$:
\begin{align}
\dfrac12\mathcal{L}|A|^2=|\nabla A|^2-\dfrac12|A|^2-|A|^4.
\end{align}
This and $n|A|^2\geq H^2>0$ imply that
\begin{align}\mathcal{L}|A|=-\dfrac12|A|-|A|^3+\dfrac{|\nabla A|^2-|\nabla |A||^2}{|A|}.
\end{align}
Since $|\nabla A|^2-|\nabla |A||^2\geq 0,$ (see, e.g. Lemma 10.2 \cite{CM})
\begin{align}\label{simons-A-1}\mathcal{L}|A|+(|A|^2+\dfrac12)|A|=\dfrac{|\nabla A|^2-|\nabla |A||^2}{|A|}\geq 0.
\end{align}
By (\ref{simons-H-1}), (\ref{simons-A-1}), and the hypothesis of the theorem,  using Lemma \ref{lem1}, it holds that   $H=C|A|$ on $\Sigma$ and 
\begin{align}\label{eq-A}\mathcal{L}|A|+(|A|^2+\dfrac12)|A|=0.
\end{align}
Hence  $|\nabla A|=|\nabla|A||$.

The rest of the proof is similar to the one of Huisken (\cite{H}. See, e.g. the proof of Theorem 0.17 in \cite{CM}). We only focus on the different points of the argument. Following the proof in \cite{CM},  $|\nabla A|=|\nabla|A||$  implies two possible cases: (I) If the rank of $A$ is greater than $2$, it implies that  $\nabla A\equiv 0$ on $\Sigma$. Thus $|A|$ is constant on $\Sigma$.   $H=C|A|>0$ says $|A|$ is a positive constant  which induces a contradiction with (\ref{eq-A}).  (II) If the rank of $A$ is $1$, then  $\Sigma$ is the product of a curve $\gamma(t)\subset \mathbb{R}^2$ and an $(n-1)$-dimensional hyperplane. If $\gamma(t)$ is a line, it contradicts with $H>0$.  If  $\gamma(t)$ is not a line,  it will contradicts the condition  (\ref{eq-A-In}),  
since the part of the weight  $e^{\frac{|x|^2}{4}}$ restricted on $\mathbb{R}^{n-1}$ has the growth order  bigger than the one of any polynomial.

\qed

Theorem   \ref{th-3} implies Corollary \ref{cor-mean} and the following

\begin{cor}\label{4-th-1} If a complete $n$-dimensional, $n\geq 2$, immersed self-expander hypersurface $\Sigma$ is mean convex, and $\int_{\Sigma}|A|^2e^{\frac{|x|^2}{4}}d\sigma<\infty$, then $\Sigma$ must be a hyperplane $\mathbb{R}^n$ through the origin.
\end{cor}

\section{The bottom of spectrum of stability operator of self-expanders}

In this section, we study the bottom $\mu_1$ of the weighted stability operator $ L=\mathcal{L}+|A|^2-\frac12=\Delta+\frac12\left<x, \nabla\cdot\right>+|A|^2-\frac12$ of a self-expander hypersurface. 

\medskip
{\it Proof of Theorem \ref{th-stab}.} Note $\mu_1$ may take $-\infty$ and in this case (\ref{stab-eq-1}) holds. Assume that $\mu_1>-\infty.$ Then the spectral theory  (see, e.g. Lemma 9.25, \cite{CM})  implies that there is a positive function $w$ such that
\[ Lw+\mu_1w=0,
\]
that is,
\[\Delta w+\frac12\left<x, \nabla w\right>+|A|^2w-\frac12 w+\mu_1w=0.
\]
Then
\begin{align}
\mathcal{L}\log w&=\dfrac{\Delta w+\dfrac12\left<x, \nabla w\right>}{w}-\dfrac{|\nabla w|^2}{w^2}\nonumber\\
&=-|A|^2+\dfrac12-\mu_1-|\nabla \log w|^2
\end{align}
Take $\phi\in \mathit{C}_0^{\infty}(\Sigma)$.
\begin{align}
\int_{\Sigma}\phi^2(\mathcal{L}\log w)e^{\frac{|x|^2}{4}}d\sigma&= \int_{\Sigma}\left(-|A|^2+(\dfrac12-\mu_1)-|\nabla \log w|^2\right)\phi^2e^{\frac{|x|^2}{4}}d\sigma\nonumber\\
\end{align}
\begin{align}\label{stab-eq-10}
\int_{\Sigma}\left(|A|^2+|\nabla \log w|^2\right)\phi^2e^{\frac{|x|^2}{4}}d\sigma&=\int_{\Sigma}(\dfrac12-\mu_1)\phi^2e^{\frac{|x|^2}{4}}d\sigma-\int_{\Sigma}\phi^2(\mathcal{L}\log w)e^{\frac{|x|^2}{4}}d\sigma\nonumber\\
&=\int_{\Sigma}(\dfrac12-\mu_1)\phi^2e^{\frac{|x|^2}{4}}d\sigma+\int_{\Sigma}\left<\nabla \phi^2, \nabla \log w\right>e^{\frac{|x|^2}{4}}d\sigma
\end{align}
By
\[\int_{\Sigma}\left<\nabla \phi^2, \nabla \log w\right>e^{\frac{|x|^2}{4}}d\sigma\leq \int_{\Sigma}|\nabla \phi|^2e^{\frac{|x|^2}{4}}d\sigma
+\int_{\Sigma}|\nabla \log w|^2\phi^2e^{\frac{|x|^2}{4}}d\sigma,
\]
(\ref{stab-eq-10}) implies that
\begin{align}\label{stab-eq-11}
\int_{\Sigma}|A|^2\phi^2e^{\frac{|x|^2}{4}}d\sigma
&\leq\int_{\Sigma}(\dfrac12-\mu_1)\phi^2e^{\frac{|x|^2}{4}}d\sigma+\int_{\Sigma}|\nabla \phi|^2e^{\frac{|x|^2}{4}}d\sigma.
\end{align}
Since the spectrum of $\mathcal{L}$ is discrete, there is the first eigenfunction $u>0$ associated to $\lambda_1$ for the operator $\mathcal{L}$ such that $u\in W^{1,2}(\Sigma)$ and 
\[\mathcal{L}u+\lambda_1 u=0.
\]
Choose $\phi=\varphi_j u$, where $\varphi_j$ are the non-negative cut-off functions   satisfying  that  $\varphi_j$ is $1$ on $B_j$,  $|\nabla\varphi|\leq 1$ on $B_{j+1}\setminus B_j$, and $\varphi=0$ on $\Sigma\setminus B_{j+1}$.  Substitute $\phi$ in (\ref{stab-eq-11}):
\begin{align}\label{stab-eq-12}
\int_{\Sigma}|A|^2\varphi_i^2u^2e^{\frac{|x|^2}{4}}d\sigma
&\leq\int_{\Sigma}(\dfrac12-\mu_1)\varphi_j^2u^2e^{\frac{|x|^2}{4}}d\sigma+\int_{\Sigma}|\nabla (\varphi_ju)|^2e^{\frac{|x|^2}{4}}d\sigma.
\end{align}
Similar to the proof of Theorem \ref{th-3},  letting $j\rightarrow \infty$ in  (\ref{stab-eq-12}),  the monotone convergence theorem implies that 
\begin{align}\label{stab-eq-15}
\int_{\Sigma}|A|^2u^2e^{\frac{|x|^2}{4}}d\sigma
&\leq\int_{\Sigma}(\dfrac12-\mu_1)u^2e^{\frac{|x|^2}{4}}d\sigma+\int_{\Sigma}|\nabla u|^2e^{\frac{|x|^2}{4}}d\sigma\nonumber\\
&=\int_{\Sigma}(\dfrac12-\mu_1+\lambda_1)u^2e^{\frac{|x|^2}{4}}d\sigma.
\end{align}
Hence 
\[\dfrac12-\mu_1+\lambda_1\geq 0,
\]
that is, (\ref{stab-eq-1}) holds:
\[   \mu_1\leq \lambda_1+\dfrac12.
\]
If  $\dfrac12+\lambda_1-\mu_1=0$, then $|A|\equiv 0$ on $\Sigma$. Thus $H\equiv 0$, and $\Sigma$ is the hyperplane through the origin. Reciprocally for the $\mathbb{R}^n$ through the origin, $L=\mathcal{L}-\frac12$. Hence
\[\mu_1=\lambda_1+\frac12=\frac{n+1}{2}.
\]

The rest is to  prove  $\mu_1\geq 1$ in (\ref{stab-eq-2}) in the case that $\Sigma$ is mean convex. It is known that  $H\equiv 0$ or $H>0$. If $H\equiv 0$,  the conclusion obviously holds.
 If $H>0$,
 the Simons' type equation (\ref{simons-H}) states $H$ is a positive solution of 
\begin{align}\label{simons-H-2}
LH+H=\mathcal{L}H+(|A|^2+\frac{1}{2})H=0.
\end{align}
This fact implies   implies that $\mu_1\geq \frac12$ (see, e.g., \cite{pw}, \cite{FS}).

\end{document}